\documentclass[12pt]{article}
\usepackage{amssymb,a4}
\usepackage[mathscr]{eucal}
\usepackage{amsthm,amsmath,anysize}

\def\l{\lambda}
\def\g{\gamma}
\def\0{\bar{0}}
\def\1{\bar{1}}
\def\e{\epsilon}

\def\g{\mathfrak{g}}
\def\L{{\Lambda}}

\def\g{{\mathfrak g}}
\def\l{{\lambda}}

\newtheorem{lemma}{Lemma}[section]
\newtheorem{theorem}[lemma]{Theorem}

\newtheorem{proposition}[lemma]{Proposition}
\newtheorem{definition}[lemma]{Definition}
\newtheorem{corollary}[lemma]{Corollary}

\hoffset \voffset \oddsidemargin=60pt \evensidemargin=45pt
\title{\bf  Reresentations of the restricted Cartan type Lie superalgebra $W(m,n,1)$}
\author{Bin Shu\footnote{Department of Mathematics,
East China Normal University, Shanghai, 200062, P. R. China} \ and Chaowen Zhang\footnote{Department of Mathematics, China university
 of Mining and Technology, Xuzhou, 221008 Jiang Su, P. R. China}}
\date{ }

\begin{document}
\maketitle

\begin{abstract}
Simple modules  for the restricted Witt superalgebra $W(m,n,1)$ are
considered. Conditions are provided for the restricted and
nonrestricted Kac modules to be simple.

\end{abstract}


\section{Introduction}
Let $\mathbf F$ be an algebraically closed field with
characteristics $p>2$. Let $\mathfrak
g=\mathfrak{g}_{\overline{0}}\oplus \mathfrak{g}_{\overline{1}}$ be
a Lie superalgebra over $\mathbf F$. $\mathfrak{g} $ is called
restricted if $\mathfrak{g}_{\bar{0}}$ is a restricted Lie algebra
and if $\mathfrak{g}_{\bar{1}}$ is a restricted
$\mathfrak{g}_{\bar{0}}$-module (see \cite{p}). The $p$-mapping
$[p]:\mathfrak{g}_{\bar{0}}\rightarrow \mathfrak{g}_{\overline{0}} $
is also called the $p$-mapping of  the Lie superalgebra
$\mathfrak{g}$.\par

Let $\mathfrak{g}$ be a restricted Lie superalgebra and $\mathfrak
M={\mathfrak M}_{\bar{0}}\oplus {\mathfrak M}_{\bar{1}}$ be a simple
$\mathfrak{g}$-module. By \cite{z},  there is a unique
$\chi\in\mathfrak{g}_{\bar{0}}^{\ast}$ such that
$(x^{p}-x^{[p]}-\chi^p(x) \cdot  1)    {\mathfrak M}=0$ for all
$x\in \mathfrak{g}_{\bar{0}}.$

A $\mathfrak{g}$-module $M$ is called having a $p$-character
$\chi\in \mathfrak{g}^*_{\overline{0}}$
 provided that
$$x^{p}\cdot m-x^{[p]}\cdot m=\chi (x) ^{p}\ m\quad \mbox{for  all}\,\, x\in
\mathfrak{g}_{\bar{0}},\, m\in M.
$$
Then each simple $\g$-module $M$ has a $p$-character $\chi$. $M$ is
called a restricted $\g$-module if $\chi=0$, and a nonrestricted
$\g$-module if $\chi\neq 0$.\par

 Cartan type Lie superalgebras $W(n), S(n), H(n), K(n)$ over the complex
 number field $\mathfrak C$ were
introduced by Kac \cite{k1}. Then Serganova has determined the
simplicity of the Kac modules for $W(n), S(n), H(n)$\cite{vs}, and
also computed their character formulas. Recently, the relative
support variety of $W(n)$ was defined and determined by
\cite{bkn}.\par In the modular case, there are more simple
restricted Lie superalgebras constructed by Liu and Zhang \cite{lz}.
According to \cite{lz}, there are four types of restricted Cartan
type Lie superalgebras, namely, Witt type, special type, Hamiltonian
type and contact type. But the structure of the last three types of
Lie superalgebras are far from explicit.  Representations of these
four types of Lie superalgebras have not been studied yet.\par The
present paper is aimed at determining the simple modules for the
restricted Witt type Lie superalgebra $W=W(m,n,1)$. The paper is
organized as
  follows. Section 2 defines  the restricted Witt Lie
  superalgebra  $W=W(m,n,1)$\cite{lz}. In Section 3 we determine the simple $W$-modules.
   Inspired by \cite{vs}, we define the
  generalized root reflections for the restricted Cartan type Lie
  superalgebras, and using which
  together with Jantzen's $u(\g)-T$ method \cite{J1} we show that the Kac module is simple if
  and only if the weight $\l$ is typical. In the case
  $n=0$, this recovers the result of Shen \cite{sh} for the restricted Lie
  algebras $W(m,1)$ in a different approach. In the case $m=0$, this is analogous to
  the result obtained in \cite{vs} for the Witt superalgebra W(n) over $\mathfrak C$.
  In Section 4, we study the case when the
  nonrestricted Kac module is simple. As a consequence, the conclusions
  about  the nonrestricted simple
  modules for the restricted Lie algebras $W(m,1)$ \cite{ho1,z1} are recovered.

\section{Preliminaries}
Let $V=V_{\overline{0}}\oplus V_{\overline{1}} $ be a
$\mathbf{Z}_{2}$-graded vector space.  For each homogeneous $x\in
V$, we use $\bar x$ to denote the grading of $x$. The algebra
$\mathrm{End}_{\mathbf{F}}(V)$ consisting of the $\mathbf{F}$-linear
transformations
 of $V$  becomes an associative superalgebra if one
defines
$$\mathrm{End}_{\mathbf{F}} (V)_{\theta }
:=\{A\in \mathrm{End}_{\mathbf{F}} (V)\mid A (V_{\mu})\subseteq
V_{\theta +\mu}, \mu\in \mathbf{Z}_{2}\}
$$
for  $\theta\in \mathbf{Z}_{2}.$ On the vector superspace
$\mathrm{End}_{\mathbf{F}} (V) =\mathrm{End}_{\mathbf{F}}
(V)_{\overline{0}}\oplus \mathrm{End}_{\mathbf{F}}
(V)_{\overline{1}}$ we define a new multiplication $[\ ,\ ]$ by
$$ [A,B]:=AB-(-1)^{\bar A\bar B}BA
\quad \mbox{for  }\ A, B\in \mathrm{End}_{\mathbf{F}} (V).$$ This
superalgebra endowed with the new multiplication is a Lie
superalgebra, denoted by
 $ {gl}(V)={gl}(V)_{\0}\oplus {gl}_{\overline{1}}(V)_{\1}.$

Let $\mathfrak{g}$ be a restricted Lie superalgebra. For each
$\chi\in \mathfrak{g}_{\0}^*$, define the $\chi$-reduced enveloping
algebra of $\mathfrak{g}$ by
$u(\mathfrak{g},\chi)=U({\mathfrak{g}})/I_{\chi}$, where $I_{\chi}$
is the $\mathbf Z_2$-graded two-sided ideal of $U({\mathfrak{g}})$
generated by elements $\{x^p-x^{[p]}-\chi(x)^p1|x\in
\mathfrak{g}_{\overline{0}}\}$. When $\chi=0,$
$u(\mathfrak{g})=:u(\mathfrak{g},0)$ is called the restricted
universal enveloping algebra of $\mathfrak{g}$. Similar to the Lie
algebra case, each $\g$-module with character $\chi$ is a
$u(\g,\chi)$-module and vise versa.\par

Let $(\mathfrak{g},[p]) $ be a restricted Lie superalgebra. Suppose
that $ e_{1},\ldots, e_{m} $ and $f_{1},\ldots,f_{n}$ are ordered
bases of $\mathfrak{g}_{\overline{0}}$ and
$\mathfrak{g}_{\overline{1}}$ respectively.  Then
$u(\mathfrak{g},\chi)$ has the following $\mathbf{F}$-basis
\cite{z}:
$$ \{ f_{1}^{b_1}\cdots f_n^{b_n}e_1^{a_1}\cdots e_m^{a_m} \mid   0\leq a_i\leq p-1;
 b_j=0 \ \mbox{or}\ 1 \}.$$

\begin{definition}\cite{sc} Let $A=A_{\0}+A_{\1}$ be a superalgebra.
Let $D\in gl(A)_{\0}\cup gl(A)_{\1}$ such that
$$D(xy)=(Dx)y+(-1)^{\bar D\bar x}x(Dy),\quad  x, y\in A_{\0}\cup A_{\1}.$$
Then $D$ is called a superderivation of $A$.
\end{definition}
 We denote by $\text{Der} (A)$ the
set of all the superderivations of $A$. Then
$\text{Der}(A)=\text{Der}(A)_{\0}\oplus \text{Der}(A)_{\1}$, where
$\text{Der}(A)_{\bar i}=\text{Der}(A)\cap gl(A)_{\bar i}$, $i=1,2$.
$\text{Der}(A)$ is a Lie superalgebra with the Lie product $$[D_1,
D_2]=D_1D_2-(-1)^{\bar D_1\bar D_2}D_2 D_1, \quad D_1,D_2\in
\text{Der}(A)_{\0}\cup \text{Der}(A)_{\1}.$$

We now define  the restricted Witt superalgebra $W(m,n,1)$
\cite{lz}. Let $m,n$ be two nonnegative integers, and let $a,b\in
\mathbf Z^m$. We write $a\leq b$ if $a_i\leq b_i$ for all $1\leq
i\leq m$ and we write $a<b$ if $a\leq b$ but $a\neq b$. If $a,b\geq
0$, define $\binom{a}{b}=\Pi^m_{i=1}\binom{a_i}{b_i}$, where
$\binom{a_i}{b_i}$ is the usual binomial coefficient with the
convention that $\binom{a_i}{b_i}=0$ unless $b_i\leq a_i$. Set
$A(m,1)=\{a\in \mathbf Z^m|0\leq a\leq \tau\}$, where
$\tau=:(p-1,\dots,p-1)$. The {\sl divided power algebra}
 $\mathfrak A(m,1)$ is an associative
$\mathbf F$-algebra having $\mathbf F$-basis $\{x^a|a\in A(m,1)\}$
and multiplication given by
$$ x^ax^b=\begin{cases} \binom{a+b}{a}x^{a+b}, &\text{if $a+b\leq
\tau$}\\0,&\text{otherwise}.\end{cases}$$  For $n\geq 0$, let $I$
denote the sequence $i_1,\dots, i_s$, where $1\leq i_1<i_2<\dots
<i_s\leq n$. Let $\mathfrak I$ be the set of all such sequences
including the empty one. For each $I\in \mathfrak I$, we let $|I|$
denote the length of  $I$.

Let $\Lambda(n)$ denote the exterior algebra with generators
$\xi_1,\xi_2,\dots,\xi_n$. For  $I=i_1,\dots, i_s$, we denote the
product $\xi_{i_1}\dots\xi_{i_s}\in\Lambda(n)$ by $\xi_I$. Then
$\Lambda(n)$ is $\mathbf Z_2$-graded with $$\Lambda(n)_{\0}=\langle
\xi_I||I| \quad\text{is even}\rangle\quad\Lambda(n)_{\1}=\langle
\xi_I||I| \quad\text{is odd}\rangle.$$ We denote
$\Lambda(m,n)=\mathfrak{A}(m,1)\otimes \Lambda(n)$. Then
$\Lambda(m,n)$ is a superalgebra with the $\mathbf{Z}_2$-gradation
given by \cite{lz}:$$ \Lambda(m,n)_{\0}=\mathfrak{A}(m,1)\otimes
\Lambda (n)_{\0},\quad
\Lambda(m,n)_{\1}=\mathfrak{A}(m,1)\otimes\Lambda(n)_{\1}.$$ We
denote the element $f\otimes g\in \Lambda(m,n)$ simply by $fg$,
$f\in\mathfrak A(m,1)$, $\g\in\Lambda(n)$.\par Define the linear map
$D_i$, $i=1,\dots,n$ and $d_j$, $j=1,\dots,m$ on $\Lambda(m,n)$ by:
$$D_i x^a\xi_I=\begin{cases} (-1)^{|i<I|}x^a\xi_{I\setminus
i},&\text{if $i\in I$}\\0,&\text{otherwise},\end{cases}\quad
d_jx^a\xi_I=x^{a-\e_j}\xi_I,$$ where $|i<I|$ is the number of
indices in $I$ that are smaller than $i$. Thus, we have $D_i\in
gl(\L(m,n))_{\1}$ and $d_j\in gl(\L(m,n))_{\0}$. It is easy to see
that
$$[D_i,D_j]=[d_i,d_j]=[D_i,d_j]=0.$$ Let $x,y\in \L(m,n)$ be
homogeneous. A short calculation shows that
$$D_i(xy)=D_i(x)y+(-1)^{\bar x}xD_i(y),\quad  d_j(xy)=d_j(x)y+xd_j(y),$$
$i=1,\dots, n$, $j=1,\dots, m$.  Therefore, $D_i\in
\text{Der}(\L(m,n))_{\1}$ and $d_j\in \text{Der}(\L(m,n))_{\0}$. Let
$$W=:W(m,n,1)=\sum^n_{i=1} \L(m,n) D_i+\sum^m_{j=1}\L(m,n) d_j\subseteq \text{Der}(\L(m,n))).$$
For each  $f\in \L(m,n)$,  clearly we have $\overline{fD_i}=\bar
f+\bar D_i$, $i=1,\dots,n$ and  $\overline{fd_i}=\bar f+\bar
d_i=\bar f$, $i=1,\dots,m$.  Let $fD,gE\in W$, where $f,g\in
\L(m,n)_{\0}\cup\L(m,n)_{\1}$ and $D,E\in\{D_1,\dots,D_n,
d_1,\dots,d_m\}$.  Then the definition of the Lie product in
$\text{Der}(\L(m,n))$ yields \cite{lz} $$
[fD,gE]=fD(g)E-(-1)^{\bar{fD}\bar{gE}}gE(f)D.$$ i.e., $W$ is Lie
sub-superalgebra of $\text{Der}(\L(m,n))$. For each $D\in
W_{\0}\subseteq \text{Der}(\L(m,n))_{\0}$, a short calculation shows
that $D^p\in W_{\0}$, so that $W$ is a restricted Lie superalgebra
with the p-map the pth power map. By \cite{lz},
$W=\text{Der}(\L(m,n))$ and $W$ is a simple  Lie superalgebra
referred to as the {\sl Witt type Lie superalgebra}.
\par
 Let $$W_i=\langle x^a\xi_ID_j||I|=i+1\rangle +\langle
x^a\xi_Id_j||I|=i\rangle.$$ Then we have $W=W_{-1}+W_0+\dots +
W_{n-1}$.  It is easy to see that $[W_i, W_j]\subseteq W_{i+j}$, so
that $W$ is $\mathbf Z$-graded. Thus, $W$ is $\mathbf Z_2$-graded
with
$$W_{\0}=\sum_{j=2k}W_j\quad   W_{\1}=\sum_{j=2k+1}W_j.$$  In case $m=0$,
 $W=W(n)$\cite{k1}. If $n=0$, then $W=W(m,1)$ is the restricted Lie
algebra of Witt type \cite{sh,st,sf}. \par By definition, $W$ has a
basis
$$\{x^a\xi_ID_j|0\leq a\leq \tau, I\in \mathfrak I, 1\leq j\leq n\}\cup\{x^a\xi_Id_j|0\leq
a\leq \tau, I\in \mathfrak I, 1\leq j\leq m\}.$$  For $x^a\xi_ID_i$,
$x^a\xi_Id_j\in W_{\0}$,  a straightforward calculation shows that
$$(x^a\xi_ID_i)^{[p]}=\begin{cases} \xi_iD_i,&\text{if $a=0$ and
$I=\{i\}$}\\0,&\text{otherwise}\end{cases} \quad
(x^a\xi_Id_j)^{[p]}=\begin{cases}(x^ad_j)^{[p]},&\text{if
$I=\phi$}\\0,&\text{otherwise.}\end{cases}$$
\par Next we introduce a new gradation for $W=W(m,n,1)$. For each $a\in A(m,1)$,
 let $|a|=\sum^m_{i=1}a_i$. We denote
$$W_{[k]}=\langle x^a\xi_ID_i,x^a\xi_Id_j|0\leq a\leq \tau, 1\leq i\leq n, 1\leq j\leq m, I\in\mathfrak I, |a|+|I|=k+1\rangle.$$
 Then we get $W=\sum^s_{i=-1}W_{[i]}$, where
$s=m(p-1)+n-1$. It is easy to see that $[W_{[i]}, W_{[j]}]\subseteq
W_{[i+j]}$. Note that each homogeneous component $W_{[i]}$
containing  elements in both $W_{\1}$ and $W_{\0}$. In particular,
$W_{[0]}$ itself is a Lie superalgebra.

\begin{lemma}$$W_{[0]}\cong gl(m|n) .$$
\end{lemma}
\begin{proof}By definition, $W_{[0]}=W^-+W_0+W^+$, where $$W^-=\langle
x_iD_j|1\leq i\leq m,1\leq j\leq n\rangle, W_0= \langle x_id_j|1\leq
i,j\leq m\rangle+\langle \xi_iD_j|1\leq i,j\leq n\rangle,$$
$$W^+=\langle \xi_id_j|1\leq i\leq n, 1\leq j\leq m\rangle.$$ Then it
is easy to see that $W_{[0]}\cong gl(m|n)$.
\end{proof}

\section{Restricted simple $\g$-modules}
Let  $W=W(m,n,1)$.  In this section we study the restricted
$W$-modules.

\begin{lemma} (\cite{z}) Let $I$ be a
 finite dimensional $\mathbf{Z}_{2}$-graded
 ideal of a Lie superalgebra $\mathfrak{g}$ and ${\mathfrak M}={\mathfrak M}_{\0}\oplus {\mathfrak M}_{\1}$ a
 simple module of $\mathfrak{g}.$ If $x$ acts nilpotently on ${\mathfrak M}$ for
 all $x\in I,$ then $ IM=0.$
 \end{lemma}
 Denote $W^{[i]}=\sum _{j\geq i}W_{[j]}$. Let $M=M_{\0}\oplus M_{\1}$ be a simple
 $u(W^{[0]})$-module. Then by the lemma above, we have $W^{[1]}M=0$. Hence
 $M$ is a simple $u(W_{[0]})$-module. Let $K(M)=u(W)\otimes
 _{u(W^{[0]})}M$ be the induced module, referred to as the restricted Kac module. The main purpose in this
 section is to determine when $K(M)$ is simple.\par
  The cartan subalgebra of $W$ is that of
  $W_{[0]}$ having basis $$\{
  \xi_1D_1,\dots,\xi_nD_n,x_1d_1,\dots,x_md_m\}.$$   We denote the dual
  basis in $H^*$:
  $$\eta_i=(\xi_iD_i)^*,  \e_j=(x_jd_j)^*, 1\leq i\leq n, 1\leq j\leq m.$$   $W$
  has a root space decomposition $W=H+\oplus
  _{\alpha \in\Delta}W_{\alpha}$ with respect to
  $H$. We denote the set of roots for the homogeneous
  component $W_{[i]}$ by $\Delta_{[i]}$.  Hence $W_{[i]}=\oplus _{\alpha\in \Delta_{[i]}
  }W_{\alpha}.$  Note that $x^a\xi_Id_j$ is a root
  vector with root $$a_1\e_1+\dots +a_m\e_m+\eta_{i_1}+\dots
  +\eta_{i_s}-\e_j,$$ while $x^a\xi_ID_j$ is a root
  vector with root $$a_1\e_1+\dots +a_m\e_m+\eta_{i_1}+\dots
  +\eta_{i_s}-\eta_j.$$  We see that if
  $\alpha \in \Delta_{[0]}$, then $dim W_{\alpha}=dim W_{-\alpha}=1$. If
  $\alpha\in\Delta_{[-1]}$, then $dimW_{\alpha}=1$ and $dimW_{-\alpha}>1$.\par
\begin{definition}Let $\g=\g_{\0}\oplus\g_{\1}$ be a Lie superalgebra
and $V=V_{\0}\oplus V_{\1}$ be a $\mathfrak{g}$-module. Let
$B=\mathfrak{h}+\mathfrak{n}^+$ be a Borel subalgebra of $\g$, where
$\mathfrak{n}^+$ is the  nilradical  and
  $\mathfrak{h}$ is a maximal torus
 with a basis $h_1,\dots, h_n$. If there
is a nonzero vector $v\in V_{\0}\cup V_{\1}$ such that $$n
v=0,\quad\text{for all}\quad n\in\mathfrak{n}^+,  h_i v=\lambda _i
v, \lambda _i\in \mathbf{F}, i=1,\dots,n,$$ then $v$ is called a
maximal vector(with respect to $B$) of weight
$(\lambda_1,\dots,\lambda_n)$.
\end{definition}
Let $\g=\mathfrak n^-+\mathfrak h+\mathfrak  n^+$ be a restricted
Lie superalgebra. Take the Borel subalgebra $B=\mathfrak h+\mathfrak
n^+$. Assume $\mathfrak h$ has basis $h_i, i=1,\dots,s$ with
$h_i^{[p]}=h_i$. Then we identify each $\l\in \mathfrak h^*$ with
the $s$-tuple $(\l(h_1),\dots,\l(h_s))$.  Let
$\l=(\l_1,\dots,\l_s)\in \mathbf F^s_p$. Then $\l$ defines a
1-dimensional restricted
 $B$-module $\mathbf F v_{\l}$: $$h
 v_{\l}=\l(h)v_{\lambda}, nv_{\lambda}=0 \quad \text{for every}\quad h\in \mathfrak h, n\in
 \mathfrak n^+.$$
  Let
 $M^{B}(\l)=u(\g)\otimes _{u(B)}\mathbf F_{\lambda}$
 be the baby verma module. Then $M^{B}(\lambda)$ has a unique
 proper maximal $\mathbf Z_2$-graded submodule. The unique simple $\mathbf Z_2$-graded quotient is denoted by
 $L^{B}(\lambda)$.
 \begin{lemma} $L^B(\l)$ has a unique
maximal vector of weight $\l$.
\end{lemma}
\begin{proof}  Let $\g=\mathfrak n^-+\mathfrak h+\mathfrak  n^+$.
 Let $v\in L^B(\l)_{\0}\cup
L^B(\l)_{\1}$ be a maximal vector. Then the simplicity of $L^B(\l)$
shows that $L^B(\l)=u(\mathfrak n^-)v$. Suppose $v'$ is also a
maximal vector in $L^B(\l)$. Then there is $f\in u(\mathfrak n^-)$
such that $v'=fv$. Since $L^B(\l)$ is simple, there is $g\in
u(\mathfrak n^-)$ such that $v=gfv$. Suppose $f\in u(\mathfrak
n^-)\mathfrak n^-$. Then $gf$ is nilpotent, so we get $v=0$, a
contradiction. Then we must have $f=c+f'$, where $0\neq c\in \mathbf
F$ and $f'\in u(\mathfrak n^-)\mathfrak n^-$. Then we get
$v'-cv=f'v$, if nonzero, is maximal. The argument above applies, one
gets $v'=cv$.

\end{proof}

Let  $M=M_{\0}+M_{\1}$ be a restricted $W$-module. Let $v\in
M_{\0}\cap M_{\1}$ be a maximal vector with respect to some Borel
subalgebra $B=H+N$, where $H$ is the maximal torus given earlier.
Let $h_i=\begin{cases} \xi_iD_i, &\text{if $1\leq i\leq
n$}\\x_{i-n}d_{i-n},&\text{if $n<i\leq n+m$.}\end{cases}$ Assume
$h_i v=\l_i v$, $i=1,\dots,m+n$.  Then since $h_i^{[p]}=h_i$,  we
have $\l_i^p-\l_i=0$ and hence $\l_i\in \mathbf F_p$.
 We denote
$$\L=\mathbf F \eta_1\oplus\dots\oplus \mathbf F \eta_n\oplus
  \mathbf F\e_1\oplus\dots\oplus\mathbf F\e_m.$$ Define the subset of $\L$
  $$\L^a=:\{a\eta_i+\eta_{i+1}+\dots+\eta_n+(p-1)\e_1+\dots
  +(p-1)\e_n|1\leq i\leq n,a\in \mathbf
  F_p\}$$$$\cup\{(p-1)\e_{j+1}+\dots +(p-1)\e_m|0\leq j\leq m\},$$
  where for $j=m$, we let $(p-1)\e_{j+1}+\dots +(p-1)\e_m=0$.
  We call  each weight $\lambda \in \L^a$ atypical,
  otherwise typical. Note that if $n=0$,  the
  atypical weights are exactly the {\sl exceptional weights} for the restricted
  Witt type Lie algebra $W(m,1)$\cite{sh,ho}; while in the case $m=0$, the
  atypical weights are  analogous  to those in the complex field case \cite{vs}. \par
  Let $W=W(m,n,1)$.  $W_{[0]}$ has a
  triangular decomposition: $W_{[0]}=N^-_{[0]}+H+N^+_{[0]}$, where
  $$N^+_{[0]}=\langle \xi_id_j|1\leq i\leq n, 1\leq j\leq
  m\rangle+\langle \xi_iD_j|1\leq i<j\leq n\rangle+\langle
  x_id_j|1\leq i<j\leq m\rangle,$$ $$N^-_{[0]}=\langle x_iD_j|1\leq
  i\leq m,1\leq j\leq n\rangle+\langle \xi_jD_i|1\leq i<j\leq
  n\rangle+\langle x_jd_i|1\leq i<j\leq m\rangle.$$ Then $W$ has
  Borel subalgebras $$B_{max}=H+N^+_{[0]}+W^{[1]}\quad
  B_{min}=W_{[-1]}+N^-_{[0]}+H.$$ Note that $W_{[-1]}=\oplus^m_{i=1} W_{-\e_i}
  \oplus \oplus^n_{j=1}W_{-\eta_j}$, where $$W_{-\e_i}=\mathbf Fd_i,
 \quad
  W_{-\eta_j}=\mathbf FD_j.$$ We now define a sequence of  root reflections $r_{\e_m}$,$\dots,
  r_{\e_1}$, $r_{\eta_n},\dots, r_{\eta_1}$. \par
  Firstly we define   $$B_m=:r_{\e_m}(B_{max})=H+N^+_m,$$  where
   $N^+_m$ is obtained by removing root spaces $$\sum_{k\geq
  1}W_{k\e_m},\quad  [N_{[0]}^-+\sum^{n}_{j=1}\mathbf FD_j+\sum^{m-1}_{i=1}\mathbf Fd_i,
   \sum_{k\geq 1}W_{k\e_m}]$$ from $N^+_{[0]}+W^{[1]}$
  and adding $W_{-\e_m}=\mathbf
  Fd_m$. Let $$N^-_m=N^-_{[0]}+\sum_{j=1}^n\mathbf FD_j+\sum^{m-1}_{i=1}\mathbf Fd_i
  +[N_{[0]}^-+\sum^{n}_{j=1}\mathbf FD_j
  +\sum^{m-1}_{i=1}\mathbf Fd_i,\sum_{k\geq 1}W_{k\e_m} ]+\sum_{k\geq 1}W_{k\e_m}.$$ Then it is easy to see
  that
  $W=N_m^-\oplus H\oplus N^+_m$ and both  $N^+_m$ and
  $N^-_m$ are  nilpotent. Therefore  $W=N_m^-+H+N^+_m$ is a new triangular
  decomposition, and hence $B_m=H+N^+_m$ is a Borel
  subalgebra.\par
  Suppose we have defined $B_{k+1}=r_{\e_{k+1}}\dots r_{\e_m}(B_{max})$, $1\leq
  k<m$. Then we define $r_{\e_k}$: $$r_{\e_k}(B_{k+1})=:B_k=H+N^+_k,$$
  where $N^+_k$ is obtained by adding $\mathbf Fd_k$ to $N^+_{k+1}$
  and removing the root spaces $$\sum_{l\geq 1}W_{l\e_k},
   [N^-_{[0]}+\sum_{j=1}^n\mathbf FD_j+\sum^{k-1}_{i=1}\mathbf Fd_i, \sum_{l\geq 1}W_{l\e_k}].$$
     Let
$$N^-_k=N^-_{[0]}+\sum_{j=1}^n\mathbf FD_j+\sum^{k-1}_{i=1}\mathbf Fd_i
  $$$$+ [N^-_{[0]}+\sum_{j=1}^n\mathbf FD_j+\sum^{k-1}_{i=1}\mathbf Fd_i,
  \sum_{l\geq 1,i\geq k}W_{l\e_i}]+\sum_{l\geq 1,i\geq k}W_{l\e_i}.$$
     Then it is easy to check that
  $W=N^-_k\oplus H\oplus N^+_k$ is a triangular decomposition  and hence $B_k$
  is a Borel subalgebra.    \par
 Recall \cite{k1,vs} $$W(n)=W(n)_{-1}+W(n)_0+\dots+W(n)_{n-1},$$ where $$W(n)_i=\langle
 \xi_ID_j|I\in\mathfrak{I}, j=1,\dots,n, |I|=i+1\rangle.$$ Let $n_0^+=\langle \xi_iD_j|1\leq i<j\leq
 n\rangle$ and $n_0^-=\langle \xi_jD_i|1\leq i<j\leq
 n\rangle$. By definition, we have
 $B_1=r_{\e_1}\dots r_{\e_m}(B_{max})=H+N^+_1$.\par  The reflection
 $r_{\eta_n}$ is defined by removing the root space $W_{\eta_n}+[n^-_0, W_{\eta_n}]$ from $N_1^+$ and
 adding $W_{-\eta_n}=\mathbf F D_n$. We denote the resulting Borel
 subalgebra  $r_{\eta_n}(B_1)$ by $B'_n$. For  each $1\leq r<n$, $r_{\eta_i}$ is
 defined by removing $W_{\eta_i}+[n^-_0,W_{\eta_i}]$ from $N_{i+1}^+$ and
 adding $W_{-\eta_i}=\mathbf F D_i$. We denote $B'_s=r_{\eta_s}\dots r_{\eta_n}(B_1)$ for
 $1\leq s\leq n$. Consequently,  we have $$B'_1=r_{\eta_1}\dots
 r_{\eta_n}r_{\e_1}\dots r_{\e_m}(B_{max})=B_{min}=H+N^+_{min},$$ where $$N^+_{min}=\langle
 D_1,\dots,D_n,d_1,\dots,d_m\rangle.$$

By the definition of $B_i$ and $B'_j$, $i=1,\dots,m$, $j=1,\dots,n$,
one can easily get
 \begin{lemma}For $0\leq s<m$(resp. $0\leq s<n$), let $v\in L^{B_{s+1}}(\l)$(resp. $L^{B'_{s+1}}(\l)$)  be the unique maximal
 vector of weight $\l$. If $d^r_sv\neq 0$(resp. $D_sv\neq 0$) but $d^{r+1}_sv=0$ for some
  $0\leq r\leq p-1$, then $d^r_sv$(resp. $D_sv$) is a maximal vector  with respect to
 the Borel subalgebra
 $B_s=r_{\e_s}(B_{s+1})$(resp. $B'_s=r_{\eta_s}(B'_{s+1})$).
 \end{lemma}
We denote $B_{m+1}=:B_{max}$ and $B'_{n+1}=:B_1$.
\begin{corollary}With the same condition as in the lemma above, we have $$L^{B_{i+1}}(\l)\cong L^{B_i}(\l-r\e_i),
 i=1,\dots m,\quad L^{B'_{s+1}}(\l)\cong L^{B'_s}(\l-\eta_s),
 s=1,\dots,n.$$
\end{corollary}
Note that
$$W_{\e_i}=\langle x^{\e_i+\e_s}d_s| 1\leq s\leq m\rangle +\langle
\xi_tx_iD_t| 1\leq t\leq n\rangle, W_{-\e_i}=\mathbf Fd_i$$ and
$$W_{\eta_i}=\langle \xi_ix^{\e_s}d_s| 1\leq s\leq m\rangle +\langle
\xi_t\xi_iD_t| t\neq i\rangle, W_{-\eta_i}=\mathbf FD_i.$$ Let
$\mathfrak{h}_{\e_i}=[W_{\e_i}, W_{-\e_i}]$ and
$\mathfrak{h}_{\eta_i}=[W_{\eta_i}, W_{-\eta_i}]$. Clearly we have
$$\mathfrak h_{\e_i}=H\quad \text{and} \quad\mathfrak h_{\eta_i}=\langle
x_id_i|1\leq i\leq m\rangle+\langle \xi_jD_j|j\neq i\rangle.$$
\begin{proposition}(1) For $1\leq i\leq m$, if $\l\neq 0, (p-1)\e_i$,
then  we have $$L^{B_{i+1}}(\l)\cong L^{B_i}(\l-(p-1)\e_i).$$ (2)
For $1\leq i\leq n$, if $\l(\mathfrak{h}_{\eta_i})\neq 0$,  then
$L^{B'_{i+1}}(\l)\cong L^{B'_i}(\l-\eta_i)$.
\end{proposition}
\begin{proof} (1) Let $v\in L^{B_{i+1}}(\l)$ be a maximal vector(with respect to $B_{i+1}$) of
weight $\l$. By the corollary above, it suffices to show that
$d_i^{p-1}v\neq 0$. We proceed with induction on $i$.\par   Assume
there is $j\neq i$ such that $\l(x_jd_j)\neq 0$. If $d^{p-1}_iv=0$,
then by applying $x^{\e_j+\e_i}d_j$, we get $$ 0=(x^{\e_j+\e_i}d_j)
d^{p-1}_iv=[x^{\e_j+\e_i}d_j,
d^{p-1}_i]v=(p-1)\l(x_jd_j)d^{p-2}_iv,$$  so that $d^{p-2}_iv=0$. By
repeated applications of $x^{\e_j+\e_i}d_j$ we get  $v=0$, a
contradiction. Similarly one gets $d^{p-1}_iv\neq 0$ if
$\l(\xi_sD_s)\neq 0$ for some $1\leq s\leq n$.\par If
$\l(\xi_sD_s)=\l(x_jd_j)=0$ for all $j\neq i$ and all $s\leq n$, we
must have $\l=a\e_i$, where $a=\l_{n+i}\in \mathbf F_p$ and $a\neq
0, p-1$. \par Suppose that $d^r_iv=0$ for some $r<p-1$. Using the
fact that $$(x^{(r+1)\e_i}d_i) d^r_iv=(-1)^r\l(x_id_i)v,$$ we get
$a=0$, so that $\l=0$,  a contradiction. It follows that
$d^{p-2}_iv\neq 0$.\par Suppose that $d^{p-1}_iv=0$. Then we get
from
$$ 0=(x^{2\e_i}d_i)
d^{p-1}_iv=[\l(x_id_i)+\frac{1}{2}(p-1)(p-2)]d^{p-2}_iv$$ that
$\l_{i+n}=\l(x_id_i)=a=p-1$, so that $\l=(p-1)\e_i$, a
contradiction.
\par (2) Let $v\in L^{B'_{i+1}}(\eta)$ be a maximal vector(with respect to $B'_{i+1}$) of
weight $\eta$. If $\l(\mathfrak{h}_i)\neq 0$, then there is $x\in
\g_{\eta_i}$ such that $\l([D_i,x])\neq 0$. From
$xD_iv=[x,D_i]v=\l([D_i,x])v\neq 0$, we get $D_iv\neq 0$, which is a
maximal vector with respect to the Borel subalgebra $B'_i$. So we
have

 $$L^{B'_{i+1}}(\l)\cong L^{B'_i}(\l-\eta_i), 1\leq i\leq n.$$
  This completes the proof.
\end{proof}
\begin{corollary} If
$\l\notin \L^a$, then $$L^{B_{max}}(\l)\cong
L^{B_{min}}(\l-\sum^m_{i=1}(p-1)\e_i-\sum^n_{j=1}\eta_j).$$
\end{corollary}
\begin{proof}Since $\l\neq 0, (p-1)\e_m$, we have $L^{B_{max}}(\l)\cong
L^{B_m}(\l-(p-1)\e_m)$ by Proposition 3.6. Since $\l\neq
(p-1)\e_{m-1}+(p-1)\e_m\in \L^a$, $\l-(p-1)\e_m\neq 0,
(p-1)\e_{m-1}$. Then Proposition 3.6 applied again, we get
$$L^{B_{max}}(\l)\cong L^{B_m}(\l-(p-1)\e_m)\cong
L^{B_{m-1}}(\l-(p-1)\e_{m-1}-(p-1)\e_m).$$ Then the corollary
follows from induction.
\end{proof}
\begin{proposition} Let $\l\in \L^a$.\par (1) If $\l=(p-1)\e_s+\dots +(p-1)\e_m$
for some $1\leq s\leq m$,
 then  $$L^{B_{max}}(\l)\cong L^{B_{s+1}}((p-1)\e_s)\cong
L^{B_s}(r\e_s)$$ for some $0<r<p-1$.\par (2) If $\l=a\eta_t+\dots
+\eta_n+\sum^m_{i=1}(p-1)\e_i$ for some   $1\leq t\leq n$, $a\in
\mathbf F_p\setminus\{0\}$, then
$$L^{B_{max}}(\l)\cong L^{B'_{t+1}}(a\eta_t)\cong L^{B'_t}(a\eta_t).$$
\end{proposition}
\begin{proof}(1) If $\l=0$, then $L^{B_{max}}(\l)$ is the 1-dimensional
trivial module. So we have $L^{B_{max}}(\l)\cong L^{B_m}(\l)$,  the
proposition holds.\par Let $\l=(p-1)\e_m$ and $v$ be the unique
maximal vector of $L^{B_{max}}(\l)$.  We claim that $d^{p-1}_mv=0$.
If this is not the case, then $v'=:d^{p-1}_mv\neq 0$, so that $v'$
is the unique maximal vector of $L^{B_{max}}(\l)\cong
L^{B_m}(\l-(p-1)\e_m)$ with respect to the Borel subalgebra $B_m$.
Since $v'$ has the weight $\l-(p-1)\e_m=0$, $L^{B_m}(\l-(p-1)\e_m)$
is the 1-dimensional trivial $W$-module. Therefore $v'$ is also a
maximal vector with respect to the Borel subalgebra $B_{max}$. But
$L^{B_{max}}(\l)$ contains a unique maximal vector $v$, a
contradiction. Assume $d^r_m v\neq 0$ but $d^{r+1}_mv=0$ for some
$r<p-1$. Since $\l\neq 0$, we have $r>0$. Then by Corollary 3.5,
$L^{B_{max}}(\l)\cong L^{B_m}(\l-r\e_m)=L^{B_m}((p-1-r)\e_m)$.\par
If $\l=(p-1)\e_i+\dots +(p-1)\e_m$, from the proof of Proposition
3.6, one gets
$$L^{B_{max}}(\l)\cong
L^{B_{i+1}}(\l-(p-1)\e_{i+1}-\dots-(p-1)\e_m)=L^{B_{i+1}}((p-1)\e_i).$$
Let $v\in L^{B_{max}}(\l)$ be the unique maximal vector with respect
to the Borel subalgebra $B_{i+1}$.  Then a similar argument as above
applied, one gets $d^{p-1}_i v=0$. Thus,
$$L^{B_{i+1}}(\l_{i+1})\cong L^{B_i}(\l_{i+1}-r\e_i)$$ for some
$0<r<p-1.$\par (2) Let $\l=a\eta_t+\dots
+\eta_n+\sum^m_{i=1}(p-1)\e_i$ for some   $1\leq t\leq n$, $a\in
\mathbf F_p\setminus\{0\}$.  Using 3.5 and 3.6, we have
$L^{B_{max}}(\l)\cong L^{B'_{t+1}}(a\eta_t)$. Let $v\in
L^{B_{max}}(\l)$ be the maximal vector with respect to the Borel
subalgebra $B'_{t+1}$ of weight $a\eta_t$.  Then $D_tv$, if nonzero,
is the maximal vector with respect to the Borel subalgebra $B'_t$.
Since $a\eta_t(\mathfrak h_{\eta_t})=0$, we get for any $x\in
W_{\eta_t}$ that $$xD_tv=a\eta_t([x,D_t])v=0.$$ \par For any $x\in
W_{\alpha}\cap B'_{t+1}$, $\alpha\neq \eta_t$,  by definition of
$B'_t$, we have $\alpha=\eta_t+\eta_j-\eta_i$ for some $i<j<t$. Then
we get
$$xD_tv=[x,D_t]v=\pm \xi_jD_i v=0.$$ The last equality is given by the
fact that $L^{B'_{t+1}}(a\eta_t)$ is also the  simple quotient of
the the induced module $u(W)\otimes _{u(\mathfrak p)}\mathbf Fv$,
where $\mathfrak p$ is the parabolic superalgebra
$B'_{t+1}+\sum_{i<j<t}\mathbf F\xi_jD_i$. \par  Then we have that
$D_tv$ is a maximal vector with respect to $B'_{t+1}$. Hence
$L^{B'_{t+1}}(a\eta_t)$ contains two maximal vectors $v$ and $D_tv$,
a contradiction. So we have $D_tv=0$. i.e., $v$ is also the maximal
vector of $L^{B'_t}(a\eta_t)$. This completes the proof.
\end{proof}

 Each simple $W$-module $L^{B_{max}}(\l)$ contains a
unique minimal vector $v'$( maximal with respect to $B_{min}$) of
weight $\l'$. Then we get \begin{corollary}$\l$ is typical if and
only if $\l'=\l-\sum^n_{i=1}\eta_i-\sum^m_{j=1}(p-1)\e_j$.
\end{corollary}
Recall $W^{[i]}=\sum_{j\geq i}W_{[j]}$. By Lemma 3.1, each simple
$u(W^{[0]})$-module $M=M_{\0}+M_{\1}$  is a simple
$u(W_{[0]})$-module. By Lemma 3.3, $M$ is generated by the unique
maximal vector $v\in M_{\0}$ of weight $\l$. We denote $M=M(\l)$ and
let $K(\l)=u(W)\otimes _{u(W^{[0]})}M(\l)$ be the induced module,
referred to as the Kac module. As a vector space, $K(\l)\cong
u(W_{[-1]})\otimes _{\mathbf F}M(\l)$. Let $Y=D_1\dots
D_nd_1^{p-1}\dots d^{p-1}_m\in u(W_{[-1]})$. It is easy to see that
$K(\l)$ contains a unique simple $u(W_{[0]})$-submodule $Y\otimes
M(\l)$.
\begin{proposition} $K(\l)$ has the unique simple quotient
$L^{B_{max}}(\l)$.
\end{proposition}
\begin{proof}Recall the baby verma module $Z(\l)=u(\g)v$
generated by the maximal vector $v$ of weight $\l$. Note that
$Z(\l)=u(B_{min})v=u(W_{[-1]})u(N^-_{[0]})v$. Since $M(\l)$ contains
the unique maximal vector (denoted also by $v$) of weight $\l$,
there is an $u(W_{[0]})$-epimorphism $\phi:
u(N^-_{[0]})v\longrightarrow M(\l).$  Since
$u(W)\otimes_{u(W^{[0]})} -$ is a right exact functor from the
category of $u(W^{[0]})$-modules to the category of $u(W)$-modules,
$\phi$ induces an epimorphism $\bar\phi: Z(\l)\longrightarrow
K(\l)$. Since $Z(\l)$ has the unique simple quotient
$L^{B_{max}}(\l)$, $K(\l)$ has the same unique simple quotient.
\end{proof}
\subsection{The category of $u(W)-T$-modules}
For  each superalgebra(resp. Lie superalgebra) $\mathscr{A}$, we
denote
$$\text{Aut}(\mathscr{A})=\{f|f\quad \text{is an automorphism of } \mathscr{A},
f(\mathscr{A}_{\bar i})=\mathscr{A}_{\bar i}, i=0,1\}.$$ Each $f\in
\text{Aut}(\mathscr{A})$ is called an even automorphism of
$\mathscr{A}$.  Inspired by \cite{cs, J1}, we introduce the
$u(W)-T$-module category in this section. \par Let
$$\text{Aut}^*(W)=\{\Phi\in \text{Aut}(W)|\Phi(W_{[i]})\subseteq
\Phi(W_{[i]}),$$$$ \Phi(W_i)\subseteq W_i, \Phi(x^{[p]})=\Phi
(x)^{[p]}, \text{for all}\quad  x\in W_{\0}\}.$$ We define two types
gradation on the super commutative algebra
 $\L(m,n)$ as follows:\par
 (1) Type I: Let $$\L(m,n)_{[i]}=\langle x^a\xi_I||a|+|I|=i, 0\leq a\leq \tau,
 I\in \mathfrak{I}\rangle.$$ Then we get $\L(m,n)=\oplus ^{n+m(p-1)}_{i=0} \L(m,n)_{[i]}$.\par
 (2) Type II: Let $$\L(m,n)_i=\langle x^a\xi_I||I|=i,I\in \mathfrak
 I\rangle.$$ Then we have $\L(m,n)=\oplus^n_{i=0}\L(m,n)_i$.\par

  Each $$\phi=:(\phi_1,\phi_2)\in GL(\mathfrak A(m,1)_{[1]})\times
  GL(\L(n)_{[1]})$$  can be extended to an element of $\text{Aut}(\L(m,n))$ which preserves both
  types of
  gradations. We denote it also by $\phi$. Now we define
 $\Phi\in \text{Aut}(W)$ by $\Phi(x)=\phi x\phi^{-1}$,  $x\in W$.
  Then it is easy to see that $\Phi(x^{[p]})=\Phi(x)^{[p]}$, for all $x\in W_{\0}$. Let $\theta
 (\phi)=\Phi$. Then $\theta$ defines
 a group homomorphism :
 $$GL(\mathfrak A(m,1)_{[1]})\times
  GL(\L(n)_{[1]})\longrightarrow
 \text{Aut}^*(W).$$
 \begin{proposition}   $\theta$ is a  monomorphism. \end{proposition}
 \begin{proof} Let $\phi=(\phi_1,\phi_2)\in GL(\mathfrak A(m,1)_{[1]})\times
  GL(\L(n)_{[1]})$ such that $\theta(\phi)=Id$. Then we have
 $\phi_1 D_i\phi_1^{-1}=D_i$ for each $i=1,\dots,n$ and
 $\phi_2d_j\phi_2^{-1}=d_j$ for each $j=1,\dots,m$. Let
 $\phi_1^{-1}(\xi_1,\dots,\xi_n)=(\xi_1,\dots,\xi_n)A$. So
 $A=(a_{ij})_{n\times n}$ is nonsingular. Then from
 $\phi_1D_i\phi^{-1}_1(\xi_j)=D_i(\xi_j)=\delta_{ij}$, we get
 $A=I_n$. This implies that $\phi_1=Id$. Similarly we
 get $\phi_2=Id$, so that $\phi=Id$. This completes the proof.

 \end{proof}

Each $\phi\in\text{Aut}(\L(m,n))$ is uniquely determined by its
action on $$\xi_1,\dots, \xi_n,x_1,\dots, x_m.$$ Then we see that
$$\{t\in \text{Aut}(\L(m,n))|t(\xi_i)=t_i\xi_i,1\leq i\leq n,
t(x_j)=t_{n+j}x_j,1\leq j\leq m,$$$$ t_i\in \mathbf F^*, i=1,\dots
,m+n\}$$ is a maximal torus of the algebraic group
$\text{Aut}(\L(m,n)$, which is isomorphic to
$$T=:\{diag(t_1,\dots,t_{m+n})|t_i\in \mathbf F^*, i=1,\dots,m+n\}.$$
Clearly $\text{Lie}(T)=H$. Let $X(T)$ be the character group of $T$.
For each $t\in T$, we  denote $\L_i(t)=t_i$, $i=1,\dots, m+n$. Then
we get $$X(T)=\mathbf Z\L_1+\dots +\mathbf Z\L_{m+n}.$$ For $t\in T$
and $h\in H$, we have
$$Adt(x^a\xi_Id_j)=(\Pi^m_{i=1}t^{a_i}_i\Pi_{l\in
I}t_{m+l})t^{-1}_jx^a\xi_Id_j=(\sum^m_{i=1}a_i\e_i+\sum_{l\in
I}\eta_l-\e_j)(t)x^a\xi_Id_j;$$
$$Adt(x^a\xi_ID_j)=(\Pi^m_{i=1}t^{a_i}_i\Pi_{l\in
I}t_{m+l})t^{-1}_{m+j}x^a\xi_ID_j=(\sum^m_{i=1}a_i\e_i+\sum_{l\in
I}\eta_l-\eta_{j})(t)x^a\xi_ID_j;$$
$$[h, x^a\xi_Id_j]=(\sum^m_{i=1}a_i\e_i+\sum_{l\in
I}\eta_l-\e_j)(h)x^a\xi_Id_j;$$
$$[h, x^a\xi_ID_j]=(\sum^m_{i=1}a_i\e_i+\sum_{l\in
I}\eta_l-\eta_j)(h)x^a\xi_ID_j.$$ Therefore the action of $T$ on $W$
coincide with that of $H$. From the definition of the p-map it is
easy to check that $Adt(x^{[p]})=Adt(x)^{[p]}$ for each $x\in
W_{\0}$.  For each $u=z_1z_2\dots z_k\in U(W) $, $z_i\in W$, we
define $Adt (u)=Adt(z_1)\dots Adt(z_k)$, then $u(W)=U(W)/I$ is also
a $T$-module, where $I$ is the two-sided ideal of $U(W)$ generated
by elements $\{x^p-x^{[p]}|x\in W_{\0}\}$. Next we define the
Jantzen's $u(\g)-T$-module category.
\begin{definition}\cite{cs,J1} Let $\g$ be a restricted Lie superalgebra with
a maximal torus $H$.  Let $$\text{Aut}^{res}(\g)=\{\Phi\in
\text{Aut}(\g)|\Phi(x^{[p]})=\Phi(x)^{[p]}, x\in \g_{\0}\}.$$ Assume
$T$ is  a diagonalizable subgroup of $\text{Aut}^{res}(\g)$  with
$\text{Lie}(T)=H$. A finite dimensional super space
$V=V_{\0}+V_{\1}$ is called a $u(\g)-T$-module, if $V$ is
$u(\g)$-module and each $V_{\bar i}$, $i=1,2$ is a $T$-module and
satisfies:\par (1) The action of $H$ coming from $\g$ and from $T$
coincide;\par (2) $t(uv)=(Adt u)v$, for $v\in V$, $t\in T$, $u\in
u(\g)$.
\end{definition}
Applying similar arguments as that in \cite{cs}, one gets that the
kernel and cokernel of a homomorphism of $u(\g)-T$-modules are also
$u(\g)-T$-modules.\par Let $\g=W$ and $T$ be the torus given
earlier. We denote by $\mathfrak {M}$  the category of
$u(W)-T$-modules. For
  $x=f\otimes m\in K(\l)=u(W_{[-1]})\otimes _{\mathbf F}M(\l)$,
  we define $tx=Adt(f)\otimes tm$, then $K(\l)\in  Obj.\mathfrak {M}$.
 Similarly we have $L^{B_i}(\l)$, $L^{B'_i}(\l)$, $Z(\l)\in Obj.\mathfrak {M}$. For each
 $M\in Obj.\mathfrak {M}$, we define the length of $M$,
 denoted $l(M)$ as the number of $Ad\bar T$ weights of $M$ minus
 one, where $\bar T=(t,\dots, t)\subseteq T$. Recall the
 sum of the negative root spaces of $W_{[0]}\cong gl(m,n)$:  $$N^-_{0}=\langle
 x_id_j|1\leq j<i\leq m\rangle +\langle \xi_jD_i|i<j\rangle +\langle
 x_iD_j|1\leq i\leq m, 1\leq j\leq n\rangle.$$  It
  is easy to see that $Ad\bar T|_{u(N^-_0)}=Id.$ Therefore the
  set of weights of $L^{B_{max}}(\l)$ in $\mathfrak M$ is $$
  \{\l(\bar T),\l(\bar T)-1,\dots,\l'(\bar T)\}.$$ It follows that $\l(\bar
  T)-\l'(\bar T)\leq n+m(p-1)$. The equality holds if and only if $\l$ is
  typical.
  \subsection{The simplicity of the Kac-module}
 Let $M=M_{\0}\oplus M_{\1}$ be a simple $u(W_{[0]})$-module, considered as
 a simple $u(W^{[0]})$-module by letting $W^{[1]}M=0$. By Lemma 3.3,  $M$ is
   generated by the
 unique maximal vector of weight $\l$. We denote $M$ by $V(\l)$.
  Recall the Kac module  $K(\l)=:K(V(\l))=u(W)\otimes _{u(W^{[0]})}V(\l)$.
  \begin{theorem} $K(\l)$ is simple if and only if $\l$
  is typical.
  \end{theorem}
\begin{proof}If $\l$ is typical, then we have
$\l'=\l-\sum^n_{i=1}\eta_i-\sum^m_{i=1}(p-1)\e_{i}$. Since  the
unique  simple $\mathbf Z_2$-graded submodule of $K(\l)$ contains
$\Pi^n_{i=1}D_i\Pi^m_{j=1}d^{p-1}_j\otimes V(\l)$, it must be
$L^{B_{min}}(\l')\cong L^{B_{max}}(\l)$. Since $K(\l)$ contains a
unique maximal vector,  we get $K(\l)=L^{B_{max}}(\l)$, so that
$K(\l)$ is simple. If $\l$ is atypical, then
$l(L(\l))<n+m(p-1)=l(K(\l))$, hence $K(\l)$ is not simple.
\end{proof}

Recall the restricted Witt algebra $L=W(m,1)=L_{-1}+L_0+\dots +L_s$,
where $L_0\cong gl(m)$ and  $H=\langle x_1d_1,\dots, x_md_m\rangle$
is its maximal torus. Let $\e_i(x_jd_j)=\delta_{ij}$, and denote
$\omega_i=(p-1)\e_{i+1}+\dots +(p-1)\e_m$, $i=0,1,\dots,m$, where
$\omega_m=0$.  The weights $\omega_0,\dots,\omega_m$ are called
exceptional weights\cite{sh,ho}.
\begin{corollary}\cite{sh}
 Denote $L^i=\sum _{j\geq i} L_j$ and let $M(\l)$ be a simple $u(L_0)$-module
 generated by the maximal vector of weight $\l$. $M(\l)$  is considered
 as a simple $u(L^0)$-module by letting $L^1M(\l)=0$. Then
the induced module $K(\l)=u(L)\otimes _{u(L^0)}M(\l)$ is simple if
and only if $\l$ is not exceptional.
\end{corollary}
Let $m=0$. Then $W(m,n,1)=W(n)$ is the restricted Lie superalgebra
given in \cite{k1, vs}. Then we get an analogues conclusion as that
in \cite{vs} in modular case.
\begin{corollary}
 Denote $W(n)^i=\sum _{j\geq i} W(n)_j$ and let $M(\l)$ be a simple $u(W(n)_0)$-module, considered
 as a simple $u(W(n)^0)$-module by letting $W(n)^1M=0$. Then
the induced module $I(\l)=u(W(n))\otimes _{u(W(n)^0)}M(\l)$ is
simple if and only if $\l\notin \L^a$, where
$$\L^a=\{a\eta_i+\eta_{i+1}+\dots +\eta_n|1\leq i\leq n, a\in \mathbf
F_p\}.$$
\end{corollary}
\begin{corollary} Let $\g=W(m,n,1)$. There are totally $p^{m+n}$ distinct(up to
isomorphism) simple $u(\g)$-modules. They are represented by
$\{L^{B_{max}}(\l)|\l\in \L\}$.
\end{corollary}
\begin{proof} Let $M=M_{\0}+M_{\1}$ be a simple $u(\g)$-module. Let
$v\in M$ be a maximal vector(with respect to $B_{max}$). It is no
loss of generality to assume $v\in M_{\0}$. Then $u(N^-_{[0]})v
=u(W_{[0]})v$ is a $\mathbf Z_2$-graded $u(W_{[0]})$-submodule
annihilated by $W^{[1]}$. Let $V$ be a simple $u(W_{[0]})$-submodule
of $u(N^-_{[0]})v$. By Lemma 3.3,  $V$ contains a unique maximal
vector of weight $\l$. Denote $V$ by $V(\l)$. Then the inclusion map
$V(\l)\longrightarrow M$ induced a $u(\g)$-epimorphism $\phi:
K(\l)\longrightarrow M$. By Proposition 3.10, we have $M\cong
L^{B_{max}}(\l)$. Since each $L^{B_{max}}(\l)$ contains a unique
maximal vector, $L^{B_{max}}(\l)\cong L^{B_{max}}(\mu)$ if and only
if $\l=\mu$ for any $\l$, $\mu\in \L$. This completes the proof.
\end{proof}

\section{Nonrestricted simple modules}
 In this section we study the simple $u(W,\chi)$-modules
with $\chi\neq 0$.
\subsection{Induced modules of type I}
Recall $W=W_{[-1]}+W_{[0]}+\dots+W_{[s]}$, $s=n+(p-1)m-1$.  For each
$-1\leq i\leq s$, let $W_{[i],\bar j}=W_{[i]}\cap W_{\bar j}$,
$j=0,1$. Then we have
$$W_{[i]}=W_{[i],\0}+W_{[i]\1}.$$  We extend the character $\chi\in
W^*_{\0}$ to $W^*$ by letting $\chi(W_{\1})=0$.  Then we define the
height of $\chi$ by $\text{ht}(\chi)=min\{i|\chi(W^{[i]})=0\}$.
Clearly we have $-1\leq \text{ht}(\chi)\leq s+1$ for each $\chi$.
\begin{definition}Let $\chi\in W^*_{\0}$ with $\text{ht}(\chi)=l+1>1$. If
there are $m+n$ elements $f_1,\dots, f_{m+n}\in W_{[l+1]}$ such that
$$\text{det}((f_1,\dots, f_{m+n})^T(D_1,\dots,D_n,d_1,\dots,d_m))\neq
0,$$ then we say that $\chi$ is nonsingular.
\end{definition}
Let $M=M_{\0}+M_{\1}$ be a simple $u(W^{[0]},\chi)$-supermodule.
Define the induced module
$$K^{\chi}(M)_1=u(W,\chi)\otimes_{u(W^{[0]},\chi)}M.$$ We call it the Kac module of
type I.  By applying a similar argument as that for \cite[Th. 2.5
]{z1}, one gets
\begin{theorem}Let $\text{ht}(\chi)=s+1>1$ and $\chi$ be
nonsingular. Then the Kac module $K^{\chi}(M)_1$ is simple.
\end{theorem}
Applying a similar as that for \cite[2.6]{z1}, one can show that
$\chi$ is nonsingular, if $\text{ht}(\chi)\leq p-1$ and
$\chi(W_{[s-1]}\cap W(m,1))\neq 0$.\par
\subsection{$\text{ht}(\chi)=1$}  Let $\g=\g_{\0}\oplus \g_{\1}$ be a
restricted Lie superalgebra. Let $\Phi\in \text{Aut}^{res}(\g)$ and
$M=M_{\0}\oplus M_{\1}$ be an $\g$-module. Denote by $M^{\Phi}$ the
$\g$-module having $M$ as its underlying vector space and
$\g$-action given by $x\cdot m=\Phi(x)m$($x\in \g, m \in M$), where
the action on the right is the given one. Then $M^{\Phi}$ is simple
if and only if $M$ is. If $M$ has character $\chi$, then $M^{\Phi}$
has character $\chi^{\Phi}$, where $\chi^{\Phi}(x)=\chi(\Phi(x))$,
$x\in \g$. By applying a similar argument as that for
\cite[4.2]{hz}, we have
\begin{lemma}Assume $\text{ht}(\chi)=1$. Let $\chi\in W_{\0}^*$ and let $\Phi\in\text{Aut}^*(W)$.
 If $M$ is a $u(W^0,\chi)$-module, then $[K^{\chi}(M)_1]^{\Phi}\cong
K^{\chi^{\Phi}}(M^{\Phi})_1$.
\end{lemma}
Recall that $W_{[0]}\cong gl(m,n)$. Then $W_{[0],\0}\cong
gl(m)\oplus gl(n)$. Also, $W$ contains $W(m,1)$ as a subalgebra. Let
$G=GL(m)\times GL(n)$. Then $\text{Lie}(G)=W_{[0],\0}$. The adjoint
action of  $G$ on $W_{[0],\0}$ can  be extended naturally to $W$, so
we have  $G\subseteq \text{Aut}^{res}(W)$. \par  Note that
$W(m,1)\subseteq W_{\0}$. For each $x\in W(m,1)^1$, clearly we have
$exp(ad x)\in \text{Aut}^{res}(W)$. Assume $\chi\in W_{\0}^*$ with
$\text{ht}(\chi)=1$. By a similar proof as that for
\cite[Th.1.2]{ho1}, one can find $\Phi\in \text{Aut}^{res}(W)$ such
that $\chi^{\Phi}(W_{[-1]})=0$ and $\chi^{\Phi}(N^+_{[0]})=0$. In
the following, we simply assume that $\chi(W_{[-1]})=0
=\chi(N^+_{[0]})$. It follows that each simple $u(W,\chi)$-module
contains a  maximal vector.
\begin{theorem} Let $\text{ht}(\chi)=1$. If $\chi(N^-_{[0]}\cap W(m,1))\neq 0$,
or $\chi(N^{\pm}_{[0]})=0$ but $\chi(H)\neq 0$.
 then $K^{\chi}(M)_1$ is simple.
\end{theorem}
\begin{proof}  Assume  $\chi(N^-_{[0]}\cap W(m,1))\neq
0$. Then there is $x_id_j$, $i>j$ such that $\chi(x_id_j)\neq 0$.
Let $$v=\sum D_Id^{a_1}_1\dots d^{a_m}_m\otimes m_{I,a}\in
K^{\chi}(M)_1$$ be a maximal vector, where each  $m_{I,a}\in M$ is
homogeneous. Applying a similar argument as that for \cite[Th.
2.4]{ho1}, one gets $v=\sum_{I} D_I\otimes m_{I,0}$. Then by
applying $x_i\xi_kd_j\in W_{[1]}$ to $v$, we get $v=1\otimes
m_{0,0}\in M$, so that $K^{\chi}(M)_1$ contains a unique maximal
vector $1\otimes m_{0,0}$ which also generates $K^{\chi}(M)_1$. It
follows that $K^{\chi}(M)_1$ is simple.\par Secondly we consider the
case $\chi(N^{\pm}_{[0]})=0$ and $\chi(H)\neq 0$. If
$\chi(x_id_i)\neq 0$ for some $1\leq i\leq m$. Then since
$\l^p_i-\l=\chi(x_id_i)^p$, $\l\notin \mathbf F_p$, so that $\l$ is
atypical. Then the proof of Theorem 3.13 implies that
$K^{\chi}(M)_1$ is simple. Similarly one gets that $K^{\chi}(M)_1$
is simple if $\chi(\xi_iD_i)\neq 0$ for some $1\leq i\leq n$.
\end{proof}
\subsection{Induced modules of type II }In this section we study the
induced nonrestricted $W$-modules with respect to a different
gradation.
\par Recall the $\mathbf Z$-grading of $W$ given in Section 2:
$W=W_{-1}+W_0+\dots + W_{n-1}$.
 $W_0=\sigma \oplus W(m,1)$, where
$$\sigma=\langle x^a\xi_iD_j|0\leq a\leq \tau, 1\leq i,j\leq
n\rangle$$ is an ideal of $W_0$. From Section 2, we have
$W^{[p]}_{2k}=0$ for all $k>0$; while in $W_0$, the p-map on
$W(m,1)$ coincides with that of the restricted  Lie algebra
$W(m,1)$. For $x^a\xi_iD_j\in\sigma$, we have
$$(x^a\xi_iD_j)^{[p]}=\begin{cases}
\xi_iD_i, &\text{if  $a=0$ and $i=j$}\\
0,&\text{otherwise.}\end{cases} $$ For each $\chi\in W^*_{\0}$,
define the height $$\text{ht}(\chi)=min\{i| \chi(W^i)=0\},$$ where
$W^i=\sum_{j\geq i}W_j$. We see that all the heights for a character
$\chi$ are odd.  It  is  easy to check that $[W_{-1},
W_{2s+1}]=W_{2s}$.  Let  $M$ be a simple $u(W^0,\chi)$-module. We
define the nonrestricted Kac module of type II
$$K^{\chi}(M)_2=: u(W,\chi)\otimes _{u(W^0,\chi)}M.$$ \begin{definition}Assume $\text{ht}(\chi)=2s+1$.  Let
$f_1,f_2,\dots,f_s$ be a basis of $W_{-1}$. If there are elements
$e_1,e_2,\dots, e_s\in W_{2s+1}$ such that
$\text{det}\chi((e_1,e_2,\dots,e_s)^T(f_1,f_2,\dots,f_s))\neq 0$,
 we say that  $\chi$ is nonsingular.
\end{definition} Applying a similar argument as that for \cite[2.5]{z1},
we get
\begin{proposition}Let $\text{ht}(\chi)=2s+1>1$.  If $\chi$ is
nonsingular, then  $K^{\chi}(M)_2$ is simple.
\end{proposition}
\subsection{$ht(\chi)=1$}
 By definition,
$$W_1=\sum^n_{i=1}\xi_iW(m,1)\oplus \mathfrak A(m,1)W(n)_1.$$  Note that $[W_{-1}, \mathfrak
A(m,1)W(n)_1]\subseteq \sigma$.
\begin{definition} Assume $\chi\in W^*_{\0}$ and $ht(\chi)=1$.
 Let $f_1,\dots, f_s$ be a ordered basis of $W_{-1}$. We say that $\chi$
is inducible if there are elements $e_1,\dots,e_s\in W_1$
satisfying:
\par (1) $\chi([e_i,f_i])\neq 0$, $i=1,\dots,s$.\par (2)  $[e_i, f_j]\in
\text{ann}
(f_{j+1}\dots f_s\otimes M)$ if $i\neq j$.\par (3)
$[[e_i,f_j],f_{j+k}]\in \text{ann} (f_{j+1}\dots f_s\otimes M)$ for
all $i,j$ and $k>0$.
\end{definition}
It is easy to see that $\chi^{\Phi}$ is inducible  for each $\Phi\in
\text{Aut}^*(W)$ if $\chi$ is inducible. \par Examples: (1) Let
$x^bD_i\in W_{-1}$ and $\xi_jx^ad_l\in W_1$. We have
$$[x^bD_i,
\xi_jx^ad_l]=\delta_{ij}\binom{a+b}{a}x^{a+b}d_l+\binom{a+b-\e_l}{b-\e_l}\xi_jx^{a+b-\e_l}D_i.$$
Let $\chi(\sigma)=0$ and $\chi(x^ad_i)=0$ for all $a\leq\tau$ and
$1\leq i\leq n$.  Assume $\chi(x^{\tau}d_l)\neq 0$ for some $1\leq
l\leq m$.  Take the ordered basis of $W_{-1}$:
$$D_1,\dots, D_n,\dots x^aD_1,\dots,x^aD_n,\dots,x^{\tau}D_1,\dots,
x^{\tau}D_n,$$ where $a's$ are put in the increasing order of $|a|$.
For $f_i=x^aD_i$, we let $e_i=\xi_ix^{\tau-a}d_l$. Then we see that
$\chi$ satisfies Definition 4.7.\par

(2) Let $\chi\in W^*_{\0}$ with
$\chi(x^{\tau}\xi_1D_1)\chi(x^{\tau}\xi_2D_2)\neq 0$ and
$\chi(x^{\tau}\xi_iD_j)=0$ if $i\neq j$. Take the same ordered basis
of $W_{-1}$ as above.  For $f_i=x^aD_i\in W_{-1}$,
 we
let $$e_i=\begin{cases} x^{\tau-a}\xi_i\xi_1D_1, &\text{if $i\neq 1$}\\
x^{\tau-a}\xi_i\xi_2D_2, &\text{if $i=1$.}\end{cases}$$ Then we see
that $\chi$ is inducible.

\begin{theorem} If $\chi\in W^*_{\0}$ be inducible, then  $K^{\chi}(M)_2$ is simple.
\end{theorem}
\begin{proof}Let $M'=M'_{\0}\oplus M'_{\1}\subseteq K^{\chi}(M)_2$ be a simple
$u(\g^0)$-submodule. Let $$ m=\sum ^s_{i=1}F_i\otimes m_i\in M'$$ be
a homogeneous nonzero element, where each $F_i$ is a product of
certain $f_i's$ and each $m_i\in M$ is homogeneous. By applying
elements $f_i\in W_{-1}$, we get $0\neq m'=\Pi^s_{i=1} f_i\otimes
m_0\in M'$ for some homogeneous $m_0\in M$. Taking elements
$e_1,\dots,e_s\in W_1$ as in Definition 4.8,  and   applying them to
$m'$,  one gets $1\otimes m_0\in M'$, hence $M\subseteq M'$. This
implies that $M'=K^{\chi}(M)_2$, so that $K^{\chi}(M)_2$ is simple.
\end{proof}

\def\refname{
\end{document}